\documentclass[12pt,a4paper]{article}
\usepackage{amssymb,amsmath,amsthm}
\newtheorem{theorem}{\indent {\bf Theorem}}
\newtheorem{lemma}{\indent {\bf Lemma}}
\theoremstyle{remark}

\begin{document}
\begin{center}
{\Large\textbf{Binary Additive Problems Involving Naturals with
Binary Decompositions of a Special Kind}}
\end{center}
\bigskip
\begin{center}
{\large \textbf{K.M. \'{E}minyan}}
\end{center}
\bigskip

\begin{abstract}
Let $h$ and $l$ be integers such that $0\le h\le 2$, $0\le l\le
4$. We obtain asymptotic formulas for the numbers of solutions of
the equations $n-3m=h$, $n-5m=l$ in positive integers $m$ and $n$
of a special kind, $m\le X$.
\end{abstract}

\bigskip
\begin{center}
\large 1. Introduction
\end{center}

Consider the binary decomposition of a positive integer $n$:
$$
n=\sum_{k=0}^\infty\varepsilon_k2^k,
$$
where $\varepsilon_k=0,1$ and $k=0,1,\ldots$

We split the set of positive integers into two nonintersecting
classes as follows:
$$
\mathbb{N}_0=\{n\in \mathbb{N},\quad
\sum_{k=0}^\infty\varepsilon_k\equiv 0\pmod 2\},\quad
\mathbb{N}_1=\{n\in \mathbb{N},\quad
\sum_{k=0}^\infty\varepsilon_k\equiv 1\pmod 2\}.
$$

In 1968, A.O. Gel'fond \cite{G} obtained the following theorem:
\textit{for the number of integers $n$, $n\le X$, satisfying the
conditions $n\equiv l\pmod m$, $n\in \mathbf{N}_j$ ($j=0,1$), the
following asymptotic formula is valid:}
\begin{equation}\label{AOG}
    T_j(X,l,m)=\frac{X}{2m}+O(X^\lambda),
\end{equation}
where $m$, $l$ are any naturals and $\lambda=\frac{\ln 3}{\ln
4}=0,7924818\ldots$

Suppose that
$$
\varepsilon(n)=\left\{%
\begin{array}{ll}
    1 & \hbox{for $n\in \mathbf{N}_0$,} \\
    -1 & \hbox{otherwise.} \\
\end{array}%
\right.
$$

The proof of formula (\ref{AOG}) is based on the estimate
\begin{equation}\label{AOG1}
    |S(\alpha)|\ll X^\lambda
\end{equation}
of the trigonometrical sum
$$
S(\alpha)=\sum_{n\le X}\varepsilon(n)e^{2\pi i \alpha n},
$$
which is valid for any real values of $\alpha$.

In 1996, author \cite{E1} proved the following theorem:
\textit{Let $F_{i,k}$ be the number of solutions of the equation
$n-m=1$, where $n\le X$, $n\in \mathbf{N}_i$, $m\in \mathbf{N}_k$,
$i,k=0,1$. Then the asymptotic formulas hold:}
$$
F_{0,0}(X)=\frac{X}{6}+O(\log X),\quad
F_{1,1}(X)=\frac{X}{6}+O(\log X),
$$
$$
F_{0,1}(X)=\frac{X}{3}+O(\log X),\quad
F_{1,0}(X)=\frac{X}{3}+O(\log X).
$$

It follows from this theorem that the orders of $F_{i,k}(X)$
strongly depend  on the values of $i$ and $k$.

In present paper we consider two problems in which the indicated
effect vanishes.

Our main results are the following theorems.

\begin{theorem}\label{T1} Let $h$ be any integer such that $0\le
h\le 2$. Let $I_{i,k}(X,h)$ be the number of solutions of the
equation $n-3m=h$, where $m\le X$, $m\in \mathbf{N}_i$, $n\in
\mathbf{N}_k$, $i,k=0,1$. Then the asymptotic formulas hold:
$$
I_{i,k}(X,h)=\frac{X}{4}+O(X^\lambda).
$$
\end{theorem}

\begin{theorem}\label{T2}
Let $l$ be any integer such that $0\le h\le 4$. Let $J_{i,k}(X,l)$
be the number of solutions of the equation $n-5m=l$, where $m\le
X$, $m\in \mathbf{N}_i$, $n\in \mathbf{N}_k$, $i,k=0,1$. Then the
asymptotic formulas hold:
$$
J_{i,k}(X,l)=\frac{X}{4}+O(X^\lambda).
$$
\end{theorem}

Let us introduce two  sums:
$$
S_3(X,h)=\sum_{n\le X}\varepsilon(n)\varepsilon(3n+h),\quad
S_5(X,l)=\sum_{n\le X}\varepsilon(n)\varepsilon(5n+l),
$$
where $h$ and $l$ are nonnegative integers.

Proofs of the theorems 1 and 2 are based on lemmas 1 and 2
consequently (see below), and also on Gel'fond's estimate
(\ref{AOG1}).
\bigskip
\begin{center}
{\large 2. Lemmas}
\end{center}

\begin{lemma}
Suppose that $h$ is an integer such that $0\le h\le 2$. The
following estimate holds:
$$
S_3(X,h)=O(\sqrt{X}).
$$
\end{lemma}

\begin{proof}[Proof]
 Grouping summands over even and over odd $n$ and using
obvious formulae $\varepsilon(2n)=\varepsilon(n)$,
$\varepsilon(2n+1)=-\varepsilon(n)$ we have the following
equalities

\begin{equation}\label{f1}
S_3(X,0)=S_3(X2^{-1},0)+S_3(X\; 2^{-1} ,1)+O(1),
\end{equation}
\begin{equation}\label{f2}
S_3(X,1)=-S_3(X\; 2^{-1} ,0)-S_3(X\; 2^{-1} ,2)+O(1),
\end{equation}
\begin{equation}\label{f3}
S_3(X,2)=S_3(X\; 2^{-1} ,1)+S_3(X\; 2^{-1} ,2)+O(1).
\end{equation}

Consider the linear combination
$$
\alpha_0S_3(X,0)+\beta_0S_3(X,1)+\gamma_0S_3(X,2),
$$
where $\alpha_0$, $\beta_0$, $\gamma_0$ are constants. By
(\ref{f1})--(\ref{f3}) we have
$$
\alpha_0S_3(X,0)+\beta_0S_3(X,1)+\gamma_0S_3(X,2)= \alpha_1S_3(X\;
2^{-1},0)+\beta_1S_3(X\; 2^{-1},1)+\gamma_1S_3(X\; 2^{-1},2)+
$$
$$
+O(|\alpha_0|)++O(|\beta_0|)++O(|\gamma_0|),
$$
where
\begin{equation*}
\left\{%
\begin{array}{ll}
    \alpha_1=\alpha_0-\beta_0, &  \\
    \beta_1=\alpha_0+\gamma_0, &  \\
    \gamma_1=\gamma_0-\beta_0. &  \\
\end{array}%
\right.
\end{equation*}

Repeating this reasoning we arrive to the equality
$$
\alpha_0S_3(X,0)+\beta_0S_3(X,1)+\gamma_0S_3(X,2)=
\alpha_jS_3(X_j,0)+\beta_jS_3(X_j,1)+\gamma_jS_3(X_j,2)+
$$
$$
+O(|\alpha_0|+\cdots +|\alpha_{j-1}|)+O(|\beta_0|+\cdots
+|\beta_{j-1}|)+O(|\gamma_0|+\cdots +|\gamma_{j-1}|),
$$
where $j$ is any integer such that $0\le j\le \log_2X-10$,
$X_j=X2^{-j}$, and the sequences $\alpha_j$, $\beta_j$, $\gamma_j$
satisfy to the system of recurrent equations
\begin{equation}\label{f4}
\left\{%
\begin{array}{ll}
    \alpha_{j+1}=\alpha_j-\beta_j, &  \\
    \beta_{j+1}=\alpha_j+\gamma_j, &  \\
     \gamma_{j+1}=\gamma_j-\beta_j. &  \\
\end{array}%
\right.
\end{equation}

Let us write (\ref{f4}) in matrix form
$$
\begin{pmatrix}
  \alpha_{j+1} \\
   \beta_{j+1} \\
   \gamma_{j+1}\\
\end{pmatrix}=\begin{pmatrix}
  1 & -1 & 0 \\
  1 & 0 & 1 \\
  0 & -1 & 1 \\
\end{pmatrix}\begin{pmatrix}
  \alpha_{j}\\
  \beta_{j} \\
  \gamma_{j} \\
\end{pmatrix}=A\begin{pmatrix}
  \alpha_{j}\\
  \beta_{j} \\
  \gamma_{j} \\
\end{pmatrix}.
$$

Then we have
$$
\begin{pmatrix}
  \alpha_{j}\\
  \beta_{j} \\
  \gamma_{j} \\
\end{pmatrix}=A^j\begin{pmatrix}
  \alpha_{0}\\
  \beta_{0} \\
  \gamma_{0} \\
\end{pmatrix}.
$$

One can easily see that $A=CBC^{-1}$, where
$$
C=\begin{pmatrix}
  1 & 1 & 1 \\
  0 & \frac{1-i\sqrt{7}}{2} & \frac{1+i\sqrt{7}}{2} \\
  -1 & 1 & 1 \\
\end{pmatrix},\quad B=\begin{pmatrix}
  1 & 0 & 0 \\
  0 & \lambda_2 & 0 \\
  0 & 0 & \lambda_3 \\
\end{pmatrix};
$$
here $1$, $\lambda_2=\frac{1+i\sqrt{7}}{2}$,
$\lambda_3=\frac{1-i\sqrt{7}}{2}$ are eigenvalues of $A$.

Thus we have
\begin{equation}\label{f5}
\begin{pmatrix}
  \alpha_{j}\\
  \beta_{j} \\
  \gamma_{j} \\
\end{pmatrix}=C\begin{pmatrix}
  1 & 0 & 0 \\
  0 & \lambda_2^j & 0 \\
  0 & 0 & \lambda_3^j \\
\end{pmatrix}C^{-1}\begin{pmatrix}
  \alpha_{0}\\
  \beta_{0} \\
  \gamma_{0} \\
\end{pmatrix}.
\end{equation}

Note that $|\lambda_2|=|\lambda_3|=\sqrt{2}$. It follows from
(\ref{f5}) that if $\begin{pmatrix}
  \alpha_{0}\\
  \beta_{0} \\
  \gamma_{0} \\
\end{pmatrix}=\begin{pmatrix}
  1 \\
  0 \\
  0 \\
\end{pmatrix}$, or $\begin{pmatrix}
  \alpha_{0}\\
  \beta_{0} \\
  \gamma_{0} \\
\end{pmatrix}=\begin{pmatrix}
  0 \\
  1 \\
  0 \\
\end{pmatrix}$, or
$\begin{pmatrix}
  \alpha_{0}\\
  \beta_{0} \\
  \gamma_{0} \\
\end{pmatrix}=\begin{pmatrix}
  0 \\
  0 \\
  1 \\
\end{pmatrix}$ then
\begin{equation}\label{f6}
|\alpha_j|=O(\sqrt{2}{\,^j}),\quad
|\beta_j|=O(\sqrt{2}{\,^j}),\quad |\gamma_j|=O(\sqrt{2}{\,^j}).
\end{equation}

Let $J$ be the largest natural number such that $ J\le
\log_2X-10$,
 $
\begin{pmatrix}
  \alpha_{0}\\
  \beta_{0} \\
  \gamma_{0} \\
\end{pmatrix}=\begin{pmatrix}
  1 \\
  0 \\
  0 \\
\end{pmatrix}$.

It follows from (\ref{f6}) that
\begin{equation*}
    |S_3(X,0)|\le |\alpha_J||S_3(X_J,0)|+|\beta_J||S_3(X_J,1)|
    +|\gamma_J||S_3(X_J,2)|+O(\sqrt{X})=O(\sqrt{X}).
\end{equation*}

The estimates $S_3(X,1)=O(\sqrt{X}\,)$, $S_3(X,2)=O(\sqrt{X}\,)$
we obtain in a similar way.
\end{proof}

\begin{lemma}
Suppose that $l$ is an integer such that $0\le l\le 4$. The
following estimate holds:
$$
S_5(X,l)=O(X^\mu),
$$
where $\mu=0,60538\ldots$
\end{lemma}
\begin{proof}[Proof]
Consider the linear combination
$$
\alpha_0S_5(X,0)+\beta_0S_5(X,1)+\gamma_0S_5(X,2)+\sigma_0S_5(X,3)
+\tau_0S_5(X,4),
$$
where $\alpha_0$, $\beta_0$, $\gamma_0$, $\sigma_0$ and $\tau_0$
are constants.

Repeating the reasoning of the proof of lemma 1 we arrive to the
equality
$$
\alpha_0S_5(X,0)+\beta_0S_5(X,1)+\gamma_0S_5(X,2)+\sigma_0S_5(X,3)
+\tau_0S_5(X,4)=
$$
$$
\alpha_JS_5(X_J,0)+\beta_JS_5(X_J,1)+\gamma_JS_5(X_J,2)
+\sigma_JS_5(X_J,3) +\tau_JS_5(X_J,4)+
$$
$$
+O(|\alpha_0|+\cdots +|\alpha_{J-1}|)+\cdots +O(|\tau_0|+\cdots
+\tau_{J-1}|),
$$
where $J$ is the largest natural number such that $ J\le
\log_2X-10$, $X_J=X2^{-J}$ and vector $\begin{pmatrix}
  \alpha_J   \\
  \beta_J \\
  \gamma_J \\
  \sigma_J \\
  \tau_J \\
\end{pmatrix}$ is defined with equation
$$
\begin{pmatrix}
  \alpha_J   \\
  \beta_J \\
  \gamma_J \\
  \sigma_J \\
  \tau_J \\
\end{pmatrix}=\begin{pmatrix}
  1 & -1 & 0 & 0 & 0 \\
  0 & 0 & 1 & -1 & 0 \\
  1 & 0 & 0 & 0 & 1 \\
  0 & -1 & 1 & 0 & 0 \\
  0 & 0 & 0 & -1 & 1 \\
\end{pmatrix}^J\begin{pmatrix}
  \alpha_0   \\
  \beta_0 \\
  \gamma_0 \\
  \sigma_0 \\
  \tau_0 \\
\end{pmatrix}=A_1^J\begin{pmatrix}
  \alpha_0   \\
  \beta_0 \\
  \gamma_0 \\
  \sigma_0 \\
  \tau_0 \\
  \end{pmatrix}.
$$

Let us write out the eigenvalues of the matrix $A_1$:
$$
\lambda_1=1,\quad \lambda_2=\frac{3^{1/3} + (9 -
\sqrt{78})^{2/3}}{3^{2/3}(9 - \sqrt{78})^{1/3}},\quad
\lambda_3=\frac{1+i\sqrt{3}}{2(27-3
\sqrt{78})^{1/3}}+\frac{(1-i\sqrt{3})(9-\sqrt{78})^{1/3}}
{2\;3^{2/3}},
$$
$$ \lambda_4=\frac{1-i\sqrt{3}}{2(27-3
\sqrt{78})^{1/3}}+\frac{(1+i\sqrt{3})(9-\sqrt{78})^{1/3}}
{2\;3^{2/3}}.
$$

Note that eigenvalues $\lambda_2$, $\lambda_3$, $\lambda_4$ are
simple and $\lambda_1$ has multiplicity 2.

It is a well known fact of Linea Algebra that there exists a
nonsingular matrix $C_1$ such that $A_1=C_1B_1C_1^{-1}$, where
$$
B_1=\begin{pmatrix}
  1 & 1 & 0 & 0 & 0 \\
  0 & 1 & 0 & 0 & 0 \\
  0 & 0 & \lambda_2 & 0 & 0 \\
  0 & 0 & 0 & \lambda_3 & 0 \\
  0 & 0 & 0 & 0 & \lambda_4 \\
\end{pmatrix}
$$

Thus we have the equality
\begin{equation}\label{f9}
    \begin{pmatrix}
  \alpha_J   \\
  \beta_J \\
  \gamma_J \\
  \sigma_J \\
  \tau_J \\
\end{pmatrix}=C_1B_1^JC_1^{-1}\begin{pmatrix}
  \alpha_0   \\
  \beta_0 \\
  \gamma_0 \\
  \sigma_0 \\
  \tau_0 \\
  \end{pmatrix}.
\end{equation}

Since $|\lambda_2|=\max\limits_{1\le
j\le4}|\lambda_j|=1,52137\ldots$ it follows from (\ref{f9}) that
the inequalities
$$
|\alpha_J|\ll |\lambda_2|^J,\ldots , |\tau_J|\ll |\lambda_2|^J.
$$
 hold and so for any $l$, $0\le l\le 4$ we have the estimate
$$
S_5(X,l)=O(X^\mu),
$$
where $\mu=\frac{\log\lambda_2}{\log 2}=0,60538\ldots$
\end{proof}

\begin{center}
{\large 3. Proofs of Theorems 1 and 2}
\end{center}
Let us prove theorem 1. For any integer $h$, $0\le h\le 2$, and
for any $i,j=0,1$ we have

\begin{gather*}
I_{i,k}(X,h)=\sum_{m\le
X}\left(\frac{1+(-1)^i\varepsilon(m)}{2}\right)
\left(\frac{1+(-1)^k\varepsilon(3m+h)}{2}\right)=\\
\frac{X}{4}+\frac{(-1)^i}{4}\sum_{m\le
X}\varepsilon(n)+\frac{(-1)^k}{4}\sum_{m\le X}\varepsilon(3m+h)+
\frac{(-1)^{i+k}}{4}S_3(X,h)+O(1)=\\
\frac{X}{4}+\frac{(-1)^i}{4}\sum_{m\le
X}\varepsilon(m)+\frac{(-1)^k}{4}\sum_{c=1}^3e^{-2\pi i
\frac{ch}{3}} \sum_{n\le 3X+h}\varepsilon(n)e^{2\pi i
\frac{cn}{3}}+\frac{(-1)^{i+k}}{4}S_3(X,h)+O(1).
\end{gather*}

Now theorem 1 follows immediately from obvious inequality
$\sum_{m\le X}\varepsilon(m)=O(1)$, Gel'fond's estimate
(\ref{AOG1}) and lemma 1.

Proof of theorem 2  essentially coincides with proof of theorem 1.
The only distinction is use lemma 2 instead of lemma 1.

\renewcommand{\refname}{References}

{\large Moscow State University of Applied Biotechnologies}

E-mail address: eminyan@mail.ru

\end{document}